\documentclass[12pt]{article}
\usepackage{amsthm,amsmath,amssymb,longtable,
MnSymbol}

\usepackage[all]{xy}

\usepackage[active]{srcltx}
\title{An informal introduction to perturbations of
matrices determined up to similarity or
congruence}
\date{}
\author{Lena Klimenko\thanks{National
Technical University of Ukraine ``Kyiv
Polytechnic Institute'', Prospect
Peremogy 37, Kiev, Ukraine.
Email: \mbox{e.n.klimenko@gmail.com}
}
    \and
Vladimir V. Sergeichuk\thanks{Institute
of Mathematics, Tereshchenkivska 3,
Kiev, Ukraine.
Supported by the FAPESP grant
2012/18139-2.
The work was done while this author
was visiting the University of S\~ao
Paulo, whose hospitality is gratefully
acknowledged. Email:
\mbox{sergeich@imath.kiev.ua}}}

\DeclareMathOperator{\im}{Im}
\DeclareMathOperator{\codim}{codim}
\DeclareMathOperator{\dddots}%
{\!\text{\raisebox{2pt}{$\cdot$}}\cdot
\text{\raisebox{-2pt}{$\cdot$}}}

\renewcommand{\ge}{\geqslant}

\newtheorem{theorem}{Theorem}[section]
\newtheorem{lemma}{Lemma}[section]
\newtheorem{corollary}{Corollary}[section]

\theoremstyle{definition}
\newtheorem{definition}{Definition}[section]

\theoremstyle{remark}
\newtheorem{remark}{Remark}[section]
\newtheorem{example}{Example}[section]

\begin{document}

\maketitle
\begin{abstract}
The reductions of a square complex
matrix $A$ to its canonical forms under
transformations of similarity,
congruence, or *congruence are unstable
operations: these canonical forms and
reduction transformations depend
discontinuously on the entries of $A$.
We survey results about their behavior
under perturbations of $A$ and about
normal forms of all matrices $A+E$ in a
neighborhood of $A$ with respect to
similarity, congruence, or *congruence.
These normal forms are called
miniversal deformations of $A$; they
are not uniquely determined by $A+E$,
but they are simple and depend
continuously on the entries of $E$.


\end{abstract}

\section{Introduction}
\label{introd}

The purpose of this survey is to give
an informal introduction into the
theory of perturbations of a square
complex matrix $A$ determined up to
transformations of similarity
$S^{-1}AS$, or congruence $S^TAS$, or
*congruence $S^*AS$, in which $S$ is
nonsingular and $S^*:=\bar S^T$.

The reduction of a matrix to its Jordan
form is an unstable operation: both the
Jordan form and a reduction
transformation depend discontinuously
on the entries of the original matrix.
For example, the Jordan matrix
$J_2(\lambda)\oplus J_2(\lambda)$ (we
denote by $J_n(\lambda )$ the $n\times
n$ upper-triangular Jordan block with
eigenvalue $\lambda $) is reduced by
arbitrarily small perturbations to
matrices
\begin{equation}\label{lis}
\left[\begin{array}{cc|cc}
\lambda&1&&\\&\lambda&&\varepsilon
\\\hline&&\lambda&1
\\ &&&\lambda
\end{array}\right]\quad\text{or}\quad
\left[\begin{array}{cc|cc}
\lambda&1&&\\&\lambda&\varepsilon &
\\\hline&&\lambda&1
\\ &&&\lambda
\end{array}\right],\qquad \varepsilon\ne
0,
\end{equation}
whose Jordan canonical forms are
$J_3(\lambda)\oplus J_1(\lambda)$ or
$J_4(\lambda)$, respectively.
Therefore, if the entries of a matrix
are known only approximately, then it
is unwise to reduce it to its Jordan
form.

Furthermore, when investigating a
family of matrices close to a given
matrix, then although each individual
matrix can be reduced to its Jordan
form, it is unwise to do so since in
such an operation the smoothness
relative to the entries is lost.

Let $J$ be a Jordan matrix.
\begin{itemize}
  \item[(a)] Arnold \cite{arn} (see
      also \cite{arn2,arn3})
      constructed a
      \emph{miniversal deformation}
      of $J$; i.e., a simple normal
      form to which all matrices
      $J+E$ close to $J$ can be
      reduced by similarity
      transformations that smoothly
      depend on the entries of $E$.

  \item[(b)] Boer and Thijsse
      \cite{den-thi} and,
      independently, Markus and
      Parilis \cite{mar-par} found
      each Jordan matrix $J'$ for
      which there exists an
      arbitrary small matrix $E$
      such that $J+E$ is similar to
      $J'$. For example, if
      $J=J_2(\lambda)\oplus
      J_2(\lambda)$, then $J'$ is
      either $J$, or
      $J_3(\lambda)\oplus
      J_1(\lambda)$, or
      $J_4(\lambda)$ with the same
      $\lambda $ (see \eqref{lis}).

      \end{itemize}

Using (b), it is easy to construct for
small $n$ the \emph{closure graph
$G_n$} for similarity classes of
$n\times n$ complex matrices; i.e., the
Hasse diagram of the partially ordered
set of similarity classes of $n\times
n$ matrices that are ordered as
follows: $a\preccurlyeq b$ if $a$ is
contained in the closure of $b$. Thus,
the graph $G_n$ shows how the
similarity classes relate to each other
in the affine space of $n\times n$
matrices.

In Section \ref{sse} we give a sketch
of constructive proof of Arnold's
theorem about miniversal deformations
of Jordan matrices, and in Sections
\ref{kkt}--\ref{jjs} we consider
closure graphs for similarity classes
and similarity bundles. In Sections
\ref{ssss} and \ref{mmj} we survey
analogous results about perturbations
of matrices determined up to congruence
or *congruence.

We do not survey the well-developed
theory of perturbations of matrix
pencils
\cite{kag,kag2,e-j-k,gar_ser,joh,k-s_triang};
i.e., of matrix pairs $(A,B)$ up to
equivalence transformations $(RAS,RBS)$
with nonsingular $R$ and $S$.

All matrices that we consider are
complex matrices.

\section{Perturbations of matrices
determined up to similarity}

\subsection{Arnold's miniversal
deformations of matrices under
similarity}\label{sse}

In this section, we formulate Arnold's
theorem about miniversal deformations
of matrices under similarity and give a
sketch of its constructive proof. Since
each square matrix is similar to a
Jordan matrix, it suffices to study
perturbations of Jordan matrices.

For each Jordan  matrix
\begin{equation}\label{kje}
J=\bigoplus_{i=1}^t(J_{m_{i1}}(\lambda _i)
\oplus\dots\oplus J_{m_{ir_i}}(\lambda _i)),
\qquad m_{i1}\ge m_{i2}\ge
\dots\ge m_{ir_i}
\end{equation}
with $\lambda _i\ne\lambda _j$ if $i\ne
j$, we define the matrix of the same
size
\begin{equation}\label{hju}
J+{\cal D}:=\bigoplus_{i=1}^t
\begin{bmatrix}
  J_{m_{i1}}(\lambda _i)+0^{\downarrow} &
  0^{\downarrow} & \dots & 0^{\downarrow}
  \\[7pt]
  0^{\leftarrow} & J_{m_{i2}}(\lambda _i)+0^{\downarrow}
   & \dddots & \vdots
  \\[7pt]
  \vdots & \dddots & \dddots & 0^{\downarrow}
  \\[7pt]
  0^{\leftarrow} & \dots & 0^{\leftarrow}
  & J_{m_{ir_i}}(\lambda _i)+
  0^{\downarrow}
\end{bmatrix}
\end{equation}
in which
\begin{equation*}\label{fde}
  0^{\leftarrow}:=
  \begin{bmatrix}
   *&0&\dots&0 \\
  \vdots&\vdots&&\vdots \\
  *&0&\dots&0 \\
  \end{bmatrix}\quad\text{and}
  \quad
   0^{\downarrow}:=
  \begin{bmatrix}
   0&\cdots&0 \\
  \vdots&&\vdots \\
  0&\cdots&0 \\
*&\cdots&* \\
  \end{bmatrix}
\end{equation*}
are blocks whose entries are zeros and
stars.

The following theorem of Arnold
\cite[Theorem 4.4]{arn} is also given
in \cite[Section 3.3]{arn2} and
\cite[\S\,30]{arn3}.

\begin{theorem}[{\cite{arn}}]\label{teo2}
Let $J$ be the Jordan matrix
\eqref{kje}. Then all matrices $J+X$
that are sufficiently close to $J$ can
be simultaneously reduced by some
transformation
\begin{equation}\label{tef}
J+X\mapsto {\cal
S}(X)^{-1}(J+X) {\cal
S}(X),\qquad\begin{matrix}
\text{${\cal S}(X)$
is analytic}\\
\text{at $0$ and } {\cal
S}(0)=I,
\end{matrix}
\end{equation}
to the form $J +{\cal D}$ defined in
\eqref{hju} whose stars are replaced by
complex numbers that depend
analytically on the entries of $X$. The
number of stars is minimal that can be
achieved by transformations of the form
\eqref{tef}, it is equal to the
codimension of the similarity class of
$J$.
\end{theorem}

The matrix \eqref{hju} with independent
parameters instead of stars is called a
\emph{miniversal deformation} of $J$
(see formal definitions in \cite{arn},
\cite{arn2}, or \cite{arn3}).

The codimension of the similarity class
of $J$ is defined as follows. For each
$A\in{\mathbb C}^{n\times n}$ and a
small matrix $X\in{\mathbb C}^{n\times
n}$,
\begin{align*}
(I-X)^{-1}A(I-X)&=(I+X+X^2+\cdots)A(I-X)\\
&= A+(XA-AX)+X(XA-AX)+X^2(XA-AX)+\cdots
\\ &=
A+\underbrace{XA-AX}
_{\text{small}}
+\underbrace{X(I-X)^{-1}(XA-AX)}
_{\text{very small}}
\end{align*}
and so the similarity class of $A$ in a
small neighborhood of $A$ can be
obtained by a very small deformation of
the affine matrix space $\{A+ XA-
AX\,|\,X\in{\mathbb C}^{n\times n}\}$.
(By the Lipschitz property \cite{rodm},
if $A$ and $B$ are close to each other
and $B=S^{-1}AS$ with a nonsingular
$S$, then $S$ can be taken near $I_n$.)

The vector space
\[
T(A):=\{XA-AX\,|\,X\in{\mathbb
C}^{n\times n}\}
\]
is the \emph{tangent space} to the
similarity class of $A$ at the point
$A$. The numbers
\begin{equation}\label{kow}
\dim_{\mathbb C} T(A),\qquad\codim_{\mathbb C} T(A):= n^2-\dim_{\mathbb C} T(A)
\end{equation}
are called the \emph{dimension} and
\emph{codimension} of the similarity
class of $A$.

\begin{remark}
The matrix \eqref{hju} is the direct
sum of $t$ matrices that are not block
triangular. But each Jordan matrix $J$
is permutation similar to some
\emph{Weyr matrix} $J^{\#}$ with the
following remarkable property: all
commuting with $J^{\#}$ matrices are
upper block triangular. Producing with
\eqref{hju} the same transformations of
permutation similarity, Klimenko and
Sergeichuk \cite{k-s_triang} obtained
an upper block triangular matrix
$J^{\#}+{\cal D}^{\#}$, which is a
miniversal deformation of $J^{\#}$.
\end{remark}

Now we show sketchily how all matrices
near $J$ can be reduced to the form
\eqref{hju} by near-identity elementary
similarity transformations; which
explains the structure of the matrix
\eqref{hju}.

\begin{lemma}\label{nbf}
Two matrices are similar if and only if
one can be transformed to the other by
a sequence of the following
transformations $($which are called
\emph{elementary similarity
transformations}; see {\rm\cite[Section
1.40]{wil}}$)$:
\begin{itemize}
  \item[\rm(i)] Multiplying column
      $i$ by a nonzero $a\in\mathbb
      C$; then dividing row $i$ by
      $a$.

  \item[\rm(ii)] Adding column $i$
      multiplied by $b\in\mathbb C$
      to column $j$; then
      subtracting row $j$
      multiplied by $b$ from row
      $i$.

  \item[\rm(iii)] Interchanging
      columns $i$ and $j$; then
      interchanging rows $i$ and
      $j$.
\end{itemize}
\end{lemma}

\begin{proof}
 Let $A$ and $B$ be similar; that is,
 $S ^{-1}AS=B$. Write $S$ as a product
 of elementary matrices:
 $S=E_1E_2\cdots E_t$. Then
 \[
A\mapsto E_1^{-1}AE_1\mapsto E_2^{-1}E_1^{-1}AE_1E_2
\mapsto\dots\mapsto E_t^{-1}\cdots E_2^{-1}E_1^{-1}AE_1E_2\cdots
 E_t=B
 \]
 is a desired sequence of elementary similarity
transformations.
\end{proof}

\begin{proof}[Sketch of the proof of
Theorem \ref{teo2}]

Two cases are possible.

\noindent {\it Case 1: $t=1$}. Suppose
first that $J=J_3(0)\oplus J_2(0)$. Let
\begin{equation}\label{nrd}
J+E=[b_{ij}]_{i,j=1}^5:=\left[\begin{array}{ccc|cc}
\varepsilon_{11}&1+\varepsilon_{12}&
\varepsilon_{13}&\varepsilon_{14}&\varepsilon_{15}
    \\
\varepsilon_{21}&\varepsilon_{22}&
1+\varepsilon_{23}&\varepsilon_{24}
&\varepsilon_{25}
    \\
\varepsilon_{31}&\varepsilon_{32}&
\varepsilon_{33}&\varepsilon_{34}
&\varepsilon_{35}
   \\\hline
\varepsilon_{41}&\varepsilon_{42}&
\varepsilon_{43}&\varepsilon_{44}&
1+\varepsilon_{45}
    \\
\varepsilon_{51}&\varepsilon_{52}&
\varepsilon_{53}&\varepsilon_{54}
&\varepsilon_{55}
\end{array}\right]
\end{equation}
be any matrix near $J$ (i.e., all
$\varepsilon _{ij}$ are small). We need
to reduce it to the form
\begin{equation}\label{nrd1}
\left[\begin{array}{ccc|cc}
0&1&0&0&0
    \\
0&0&1&0&0
    \\
*&*&*&*&*
   \\\hline
*&0&0&0&1
    \\
*&0&0&*&*
\end{array}\right],
\end{equation}
in which the $*$'s are small complex
numbers, by those transformations from
Lemma \ref{nbf} that are close to the
identity transformation.

Dividing column 2 of \eqref{nrd} by
$1+\varepsilon_{12}$ and multiplying
row 2 by $1+\varepsilon_{12}$
(transformation (i)), we make
$b_{12}=1$. Since $\varepsilon_{12}$ is
small, this transformation is
near-identity and the obtained matrix
is near $J$. Some $b_{ij}$ and
$\varepsilon _{ij}$ have been changed,
but we use the same notation for them.

Subtracting column 2 (with $\varepsilon
_{12}=0$) multiplied by
$\varepsilon_{11}$ from column 1, we
make $b_{11}=0$; the inverse
transformation of rows (which must be
done by the definition of
transformation (ii)) slightly changes
row 2. Analogously, we make
$b_{13}=b_{14}=b_{15}=0$ subtracting
column 2; the inverse transformations
of rows slightly change row 2.

We obtain
\[
[b_{ij}]_{i,j=1}^5=\left[\begin{array}{ccc|cc}
0&1&0&0&0
    \\
\varepsilon_{21}&\varepsilon_{22}&
1+\varepsilon_{23}&\varepsilon_{24}
&\varepsilon_{25}
    \\
\varepsilon_{31}&\varepsilon_{32}&
\varepsilon_{33}&\varepsilon_{34}
&\varepsilon_{35}
   \\\hline
\varepsilon_{41}&\varepsilon_{42}&
\varepsilon_{43}&\varepsilon_{44}&
1+\varepsilon_{45}
    \\
\varepsilon_{51}&\varepsilon_{52}&
\varepsilon_{53}&\varepsilon_{54}
&\varepsilon_{55}
\end{array}\right]
\]
with row 1 as in \eqref{nrd1}. In the
same manner, we make $b_{23}=1$
dividing column 3 by
$1+\varepsilon_{23}$, and then
$b_{21}=b_{22}=b_{24}=b_{25}=0$
subtracting column 3 (transformations
(i) and (ii)); the inverse
transformations with rows slightly
change row 3. In the obtained matrix,
we make $b_{45}=1$; then
$b_{41}=b_{42}=b_{43}=b_{44}=0$; the
inverse transformations with rows
slightly change row 5.

We have obtained a matrix of the form
\begin{equation*}\label{nrd2}
\left[\begin{array}{ccc|cc}
0&1&0&0&0
    \\
0&0&1&0&0
    \\
*&*&*&*&*
   \\\hline
0&0&0&0&1
    \\
*&*&*&*&*
\end{array}\right]\qquad
(\text{$*$'s are small numbers})
\end{equation*}
by using near-identity elementary
similarity transformations with
\eqref{nrd}.

To reduce the number of stars, we
subtract row 2 multiplied by $b_{53}$
from row 5 making $b_{53}=0$. The
inverse transformation of columns adds
column 5 multiplied by the old $b_{53}$
to column 2. Then we make
$b_{42}=b_{52}=0$ using row 1; the
inverse transformations of columns
slightly change $b_{31}$, $b_{41}$, and
$b_{51}$.

We have simultaneously reduced all
matrices \eqref{nrd} near $J$ to the
form \eqref{nrd1} by a similarity
transformation that analytically
depends on all $\varepsilon _{ij}$ and
that is identity if all $\varepsilon
_{ij}=0$.

In the same manner, all matrices
$J(0)+E$ near a nilpotent Jordan matrix
\[
J(0):=J_{m_{1}}(0)
\oplus\dots\oplus J_{m_{r}}(0),
\qquad m_{1}\ge m_{2}\ge
\dots\ge m_{r}
\]
can be reduced first to matrices of the
form
\[
\begin{bmatrix}
J_{m_{1}}(0)+0^{\downarrow}&\dots&0^{\downarrow}
\\\vdots&\dddots&\vdots\\
0^{\downarrow}&\dots&J_{m_{r}}(0)+0^{\downarrow}
\end{bmatrix}
\]
and then to matrices of the form
\eqref{hju} with $t=1$, $\lambda _1=0$,
and $m_1,\dots,m_r$ instead of
$m_{11},\dots,m_{1r_1}$.

This proves the theorem for each Jordan
matrix $J(\lambda)=J(0)+\lambda I$ with
a single eigenvalue $\lambda $ since
$S(E)^{-1}(J(\lambda)
+E)S(E)=S(E)^{-1}(J(0)+E)S(E)+\lambda
I$.
\medskip

\noindent {\it Case 2: $t\ge 2$}.  In
this case, \eqref{kje} has distinct
eigenvalues. Write \eqref{kje} in the
form $J=J_1\oplus\dots\oplus J_t$,
where each $J_i:=J_{m_{i1}}(\lambda _i)
\oplus\dots\oplus J_{m_{ir_i}}(\lambda
_i)$ is of size $n_i\times n_i$ and has
the single eigenvalue $\lambda _i$. Let
\begin{equation}\label{hte}
J+E=\begin{bmatrix}
J_1+E_{11}&\dots&E_{1t} \\
\vdots&\dddots&\vdots\\
E_{t1}&\dots&J_t+E_{tt} \\
    \end{bmatrix}
\end{equation}
be any matrix near $J$ (i.e., all
$E_{ij}$ are small). We make $E_{ij}=0$
for all $i\ne j$ by near-identity
similarity transformations as follows.

Represent \eqref{hte} in the form
$J+E^{\Swarrow}+ E^{\Nearrow} $ in
which
\[
J+E^{\Swarrow}:=\begin{bmatrix}
J_1&&&0 \\
E_{21}&J_2\\
\vdots&\ddots&\ddots\\
E_{t1}&\dots&E_{t,t-1}&J_t\\
    \end{bmatrix},\quad
E^{\Nearrow}:=
\begin{bmatrix}
E_{11}&E_{12}&\dots&E_{1t} \\
&E_{22}&\ddots&\vdots\\
&&\ddots&E_{t-1,t}\\
0&&&E_{tt}\\
    \end{bmatrix}.
\]

Let us reduce $J+E^{\Swarrow}$. Add to
its first vertical strip the second
strip multiplied by any $n_2\times n_1$
matrix $M$ to the right. Make the
inverse transformation of rows:
subtract from the second horizontal
strip the first strip multiplied by $M$
to the left. This similarity
transformation replaces $E_{21}$ with
$E_{21}+ J_2M-MJ_1$. Since $J_1$ and
$J_2$ have distinct eigenvalues, there
exists $M$ for which $E_{21}+
J_2M-MJ_1=0$ (see \cite[Chapter VIII,
\S\,3]{gan}). Moreover, $M$ is small
since $E_{21}$ is small.

In the same manner, we successively
make zero the other blocks of the first
underdiagonal
$E_{21},E_{32},\dots,E_{t,t-1}$ of
$J+E^{\Swarrow}$, then the blocks of
its second underdiagonal
$E_{31},\dots,E_{t,t-2}$, and so on.
Thus, there exists a near-identity
matrix $S_1$ such that
$S_1^{-1}(J+E^{\Swarrow})S_1=J_1\oplus\dots\oplus
J_t$.

We make the same similarity
transformation with the whole matrix
$J+E=J+E^{\Swarrow}+ E^{\Nearrow} $ and
obtain the matrix
$J+E':=S_1^{-1}(J+E)S_1$. Its
underdiagonal blocks $E'_{ij}$ ($i>j$)
coincide with the underdiagonal blocks
of $S_1^{-1}E^{\Nearrow}S_1$, which are
very small since all $E_{ij}$ are small
and the transformation is
near-identity.

We apply the same reduction to $J+E'$
and obtain a matrix
$J+E''=S_2^{-1}(J+E')S_2$ whose
underdiagonal blocks $E''_{ij}$ ($i>j$)
are very very small, and so on.

The infinite product $S_1S_2\dots$
converges to a near-identity matrix $S$
such that all underdiagonal blocks of
$J+\tilde E:=S^{-1}(J+E)S$ are zero.

By near-identity similarity
transformations, we successively make
zero the first overdiagonal $\tilde
E_{12},\tilde E_{23},\dots,\tilde
E_{t-1,t}$ of $J+\tilde E$, then its
second overdiagonal $\tilde
E_{13},\dots,\tilde E_{t-2,t}$, and so
on.

We have reduced \eqref{hte} to the
block diagonal form
$(J_1+F_1)\oplus\dots\oplus (J_t+F_t)$
in which all $F_i$ are small. Reducing
each summand $J_i+F_i$ as in Case 1, we
obtain a matrix of the form
\eqref{hju}.
\end{proof}

\begin{remark}
In the above proof we have described
sketchily how to construct the
transformation \eqref{tef}. Algorithms
for constructing this transformation
are discussed in \cite{mai,mai1}.
\end{remark}

\subsection{Change of the Jordan
canonical form by arbitrarily small
perturbations}\label{kkt}

Let $J$ be a Jordan matrix and let
$\lambda $ be its eigenvalue. Denote by
$w_{\lambda j}$ the number of Jordan
blocks $J_m(\lambda)$ of size $m\ge j$
in $J$; the sequence $(w_{\lambda 1},
w_{\lambda 2},\dots)$ is called the
\emph{Weyr characteristic} of $J$  (and
of any matrix that is similar to $J$)
for the eigenvalue $\lambda$.

The following theorem was proved by
Boer and Thijsse \cite{den-thi} and,
independently, by Markus and Parilis
\cite{mar-par}; another proof was given
by Elmroth, Johansson, and
K\r{a}gstr\"{o}m \cite[Theorem
2.2]{kag2}.

\begin{theorem}[{\cite{den-thi,mar-par}}]
\label{kjd} Let $J$ and $J'$ be Jordan
matrices of the same size. Then $J$ can
be transformed to a matrix that is
similar to $J'$ by an arbitrarily small
perturbation if and only if $J$ and
$J'$ have the same set of eigenvalues
with the same multiplicities, and their
Weyr characteristics satisfy
\[
w_{\lambda 1}\ge w'_{\lambda 1},\quad
w_{\lambda 1}+w_{\lambda 2}\ge w'_{\lambda 1}+w'_{\lambda 2},\quad
w_{\lambda 1}+w_{\lambda 2}+w_{\lambda 3}\ge
w'_{\lambda 1}+w'_{\lambda 2}+w'_{\lambda 3},
\ \dots
\]
for each eigenvalue $\lambda $.
\end{theorem}

Theorem \ref{kjd} was extended to
Kronecker's canonical forms of matrix
pencils by Pokrzywa \cite{pok}.

\subsection{Closure graphs for
similarity classes}\label{sso}

\begin{definition}\label{lel}
Let $T$ be a topological space with an
equivalence relation. The \emph{closure
graph} (or \emph{closure diagram}) is
the directed graph whose vertices
bijectively correspond to the
equivalence classes and for equivalence
classes $a$ and $b$ there is a directed
path from a vertex of $a$ to a vertex
of $b$ if and only if $a\subset
\overline b$, in which $\overline b$
denotes the closure of $b$.
\end{definition}

Thus, the closure graph is the Hasse
diagram of the set of equivalence
classes with the following partial
order: $a\preccurlyeq b$ if and only if
$a\subset \overline{b}$. The closure
graph shows how the equivalence classes
relate to each other in $T$.

In this section, $T=\mathbb C^{n\times
n}$ and the equivalence relation is the
similarity of matrices. Since each
similarity class contains exactly one
Jordan matrix determined up to
permutations of Jordan blocks, we
identify the vertices with the Jordan
matrices determined up to permutations
of Jordan blocks.

Theorem \ref{kjd} admits to construct
the closure graphs due to the following
lemma.

\begin{lemma}\label{mjt}
The closure graph for similarity
classes of $n\times n$ matrices
contains a directed path from a Jordan
matrix $J$ to a Jordan matrix $J'$ if
and only if $J$ can be transformed to a
matrix that is similar to $J'$ by an
arbitrarily small perturbation.
\end{lemma}

\begin{proof} Denote by $[M]$ the
similarity class of a square matrix
$M$.

``$\Longleftarrow$'' Let $J$ can be
transformed to a matrix that is similar
to $J'$ by an arbitrarily small
perturbation. Then there exists a
sequence of matrices $J+E_1,\ J+E_2,\
J+E_3,\,\dots$ in $[J']$ that converges
to $J$. This means that
$J\in\overline{[J']}$. Let $A\in [J]$;
i.e., $A=S^{-1}JS$ for some $S$. Then
the sequence of matrices
$S^{-1}(J+E_i)S=A+S^{-1}E_iS$
($i=1,2,\dots$) in $[J']$ converges to
$A$, and so $A\in\overline{[J']}$.
Therefore, $[J]\subset\overline{[J']}$
and there is a directed path from $J$
to $J'$.
\end{proof}

\begin{corollary}
By Theorem \ref{kjd}, the arrows are
only between Jordan matrices with the
same sets of eigenvalues. Let $J$ be a
Jordan matrix.

\begin{itemize}
\item Let $J'$ be a Jordan matrix
    of the same size. Each
    neighborhood of $J$ contains a
    matrix whose Jordan canonical
    form is $J'$ if and only if
there is a directed path from
    $J$ to $J'$ $($if $J=J'$ then
    there always exists the
    ``lazy''
      path of length $0$ from $J$
    to $J')$.

  \item The closure of the
      similarity class of $J$ is
      equal to the union of the
      similarity classes of all
      Jordan matrices $J'$ such
      that there is a directed path
      from $J'$ to $J$ $($if $J=J'$
      then  there always exists the
  ``lazy''
      path$)$.
\end{itemize}
\end{corollary}

\begin{example}\label{ex2}
Let us construct the closure graph for
similarity classes of ${4\times 4}$
matrices. Each Jordan matrix is a
direct sum of Jordan blocks
$J_m(\lambda )$. Replacing them by
$\lambda^m$ and deleting the symbols
$\oplus$, we get the compact notation
of Jordan matrices which was used by
Arnold \cite{arn}. For example,
\begin{equation}\label{mjw}
\lambda^2\lambda\mu\quad\text{is}\quad
J_2(\lambda )\oplus J_1(\lambda)\oplus
J_1(\mu)
\end{equation}
(we write $\lambda,\mu $ instead of
$\lambda ^1,\mu^1$).

For all Jordan matrices of size
${4\times 4}$ with eigenvalue $0$, we
have
\begin{center}
\begin{tabular}{|c|c|c|}
\hline
 Jordan  & its Weyr characte-
& the sequence $(w_1,w_1+w_2,$\\
matrix& ristic  $(w_1,w_2,w_3,w_4)$&
$w_1+w_2+w_3,w_1+w_2+w_3+w_4)$
\\
\hline
0000 & (4,0,0,0)&(4,4,4,4)\\
%\hline
$0^200$ & (3,1,0,0)&(3,4,4,4)\\
%\hline
$0^20^2$ & (2,2,0,0)&(2,4,4,4)\\
%\hline
$0^30$ & (2,1,1,0)&(2,3,4,4)\\
%\hline
$0^4$ & (1,1,1,1)&(1,2,3,4)\\
\hline
\end{tabular}
\end{center}

Using this table, Theorem \ref{kjd},
and Lemma \ref{mjt}, it is easy to
construct the following closure graph
for similarity classes of nilpotent
$\mathbf{4\times 4}$ matrices:
\begin{equation*}\label{lbt}
0000\to 0^200\to 0^20^2\to
0^30\to 0^4
\end{equation*}

In the same way, one can construct the
closure graph for similarity classes of
all $4\times 4$ matrices, which is
presented in Figure \ref{kib1}.
\begin{figure}[hbt]
\begin{equation}\label{ky}
\begin{split}
\xymatrix@C=8pt@R=8pt{
{\lambda^4}&{\lambda^3\mu}
&{\lambda^2\mu^2}
&{\lambda^2\mu\nu}
&{\lambda\mu\nu\xi}&&
   {\text{dimension }12}
\\
{\lambda^3\lambda }\ar[u]
&{\lambda^2\lambda\mu}\ar[u]
&{\lambda^2\mu\mu}\ar[u]
&{\lambda\lambda \mu\nu }\ar[u]&&&
   {\text{dimension }10}
\\
{\lambda ^2\lambda ^2}\ar[u]
&&{\lambda\lambda\mu\mu}\ar[u]
&&&&
   {\text{dimension }8}
\\
{\lambda ^2\lambda \lambda }\ar[u]
&{\lambda \lambda \lambda \mu}\ar[uu]
&&&&&
   {\text{dimension }6}
\\
{\lambda \lambda \lambda \lambda }\ar[u]
&&&&&&
   {\text{dimension }0}
      }
\end{split}
\end{equation}
\caption{\small\it The closure graph
for similarity classes of
$4\times 4$
matrices}\label{kib1}
\end{figure}
The graph is infinite:
$\lambda,\mu,\nu,\xi$ are arbitrary
distinct complex numbers. The
similarity classes of  $4\times 4$
Jordan matrices $J$ that are located at
the same horizontal level in \eqref{ky}
have the same dimension (defined in
\eqref{kow}), which is indicated to the
right and is calculated as follows: it
equals $16-\codim_{\mathbb C} T(J)$, in
which $\codim_{\mathbb C} T(J)$ is the
number of stars in \eqref{hju} (see
\eqref{kow} and Theorem \ref{teo2}).
For example, if $J$ is \eqref{mjw} with
$\lambda \ne\mu $, then \eqref{hju} is
\[
\left[\begin{array}{cc|c|c}
\lambda &1&0&0
    \\
*&\lambda+* &*&0
    \\ \hline
*&0&\lambda +*&0
   \\\hline
0&0&0&\mu+*
\end{array}\right]
\]
and so $\dim_{\mathbb
C}(J)=16-\codim_{\mathbb C} T(J)
=16-6=10.$
\end{example}

The following example shows that the
structure of the closure graph for
larger matrices is not so simple as in
\eqref{ky}.

\begin{example}\label{ex1}
The closure graph for similarity
classes of $6\times 6$ nilpotent
matrices is presented in Figure
\ref{kib}.
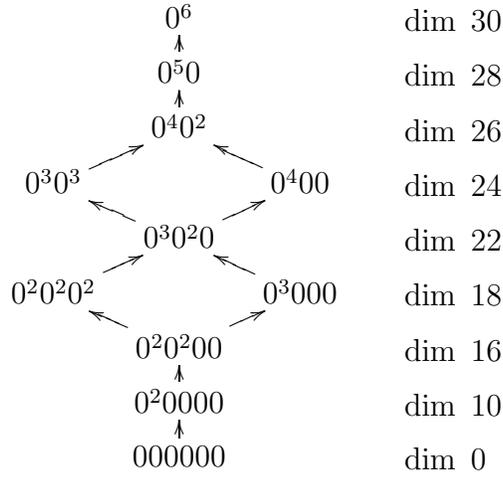
\begin{figure}[hbt]
\begin{equation*}
\xymatrix@R=6pt@C=8pt{
 &{0^6}&&
 *[r]{\dim\ 30}
    \\
 &{0^50}\ar[u]&&
 *[r]{\dim\ 28}
    \\
 &{0^40^2}\ar[u]&&
 *[r]{\dim\ 26}
    \\
{0^30^3}\ar[ur]&&{0^400}
\ar[ul]& *[r]{\dim\
24}
    \\
 &{0^30^20}\ar[ul]\ar[ur]&&
 *[r]{\dim\ 22}
    \\
{0^20^20^2}\ar[ur]&&{0^3000}\ar[ul]&
*[r]{\dim\ 18}
    \\
 &{0^20^200}\ar[ul]\ar[ur]&&
 *[r]{\dim\ 16}
    \\
&{0^20000}\ar[u] & &
*[r]{\dim\ 10}
      \\
&{000000}\ar[u]& &
*[r]{\dim\ 0}
}
\end{equation*}
\caption{\small\it The closure graph
for similarity classes of
$6\times 6$ nilpotent
matrices}\label{kib}
\end{figure}
This graph was taken from \cite[Figures
3 and 22]{joh}, where P. Johansson
describes the StratiGraph, which is a
software tool for constructing the
closure graphs for similarity classes
of matrices, for strict equivalence
classes of matrix pencils, and for
bundles of matrices and pencils (see
Section \ref{jjs} about bundles and the
web page
\begin{center}
http://www.cs.umu.se/english/research/groups/matrix-computations/stratigraph/
\end{center}
about the StratiGraph).
\end{example}

\subsection{Closure graphs for
similarity bundles}\label{jjs}

Arnold \cite[\S\,5.3]{arn} defines a
\emph{bundle of matrices under
similarity} as a set of all matrices
having the same \emph{Jordan type},
which is defined as follows: matrices
$A$ and $B$ have the same Jordan type
if there is a bijection from the set of
distinct eigenvalues of $A$ to the set
of distinct eigenvalues of $B$ that
transforms the Jordan canonical form of
$A$ to the Jordan canonical form of
$B$. For example, the Jordan matrices
\begin{equation*}\label{lic}
J_3(0)\oplus J_2(0)\oplus
J_5(1),\qquad J_3(2)\oplus J_2(2)\oplus
J_5(-3)
\end{equation*}
belong to the same bungle. All matrices
of a bundle have similar properties and
not only with respect to perturbations;
for example, its Jordan canonical
matrices have the same set of commuting
matrices.

Note that the closure graph for bundles
of $n\times n$ matrices under
similarity has a finite number of
vertices; moreover, it is in some sense
more informative than the closure graph
for similarity classes. For example,
one cannot see from the latter graph
that each neighborhood of $J_n(\lambda
)$ contains a matrix with $n$ distinct
eigenvalues (since there is no diagonal
matrix whose similarity class has a
nonzero intersection with each
neighborhood of $J_n(\lambda )$). But
the closure graph for bundles has an
arrow from the bundle containing
$J_n(\lambda )$ to the bundle of all
matrices with $n$ distinct eigenvalues.

Furthermore, not every convergent
sequence of $n\times n$ matrices
\begin{equation}\label{ktc}
B_1,B_2,\ldots\to A,
\end{equation}
in which all $B_i$ are not similar to
$A$, gives a directed path in the
closure graph for similarity classes.
But every sequence \eqref{ktc}, in
which all $B_i$ do not belong to the
bundle $\cal A$ that contains $A$,
gives at least one directed path in the
closure graph for similarity bundles.
Indeed, the number of bundles of
$n\times n$ matrices is finite, and so
there is an infinite subsequence
$B_{n_1},B_{n_2},\ldots\to A$ in which
all $B_{n_i}$ belong to the same bundle
$\cal B$. Hence $A\in\overline{\cal
B}$. One can prove that ${\cal
A}\subset\overline{\cal B}$.

\begin{example}\label{ex3}
The closure graph for similarity
bundles of ${4\times 4}$ matrices is
presented in Figure \ref{kib3} (it is
given in another form in Johansson's
guide \cite[Figure 24]{joh}).
\begin{figure}[hbt]
\begin{equation}\label{nkt}
\begin{split}
\xymatrix@C=3pt@R=3pt{
&&&&{\lambda\mu\nu\xi}&
{\dim\ 16}
\\
&&&{\lambda^2\mu\nu }\ar[ur]&&
{\dim\ 15}
\\
&{\lambda^3\mu }\ar[urr]&
{\lambda^2\mu^2}\ar[ur]&&&
{\dim\ 14}
\\
{\lambda^4}\ar[ur]\ar[urr]&&&
{\lambda\lambda \mu\nu }\ar[uu]&
{\dim\ 13}
\\
&{\lambda^2\lambda\mu}\ar[uu]\ar[urr]&
{\lambda^2\mu \mu}\ar[uu]\ar[ur]&&
{\dim\ 12}
\\
{\lambda^3\lambda
}\ar[uu]\ar[ur]\ar[urr]&&&&
{\dim\ 11}
\\
&&{\lambda\lambda\mu\mu}\ar[uu]&&
{\dim\ 10}
\\
{\lambda ^2\lambda ^2}\ar[uu]\ar[urr]&&
&&{\dim\ 9\ }
\\
&{\lambda \lambda \lambda \mu}\ar[uuuu]
&&{\dim\ 8}
\\
{\lambda ^2\lambda \lambda
}\ar[ur]\ar[uu]&&& {\dim\ 7}
\\
{\lambda \lambda \lambda \lambda
}\ar[u]&&& {\dim\ 1}
      }
\end{split}
\end{equation}
\caption{\small\it The closure graph for similarity
bundles of ${4\times 4}$ matrices}
\label{kib3}
\end{figure}
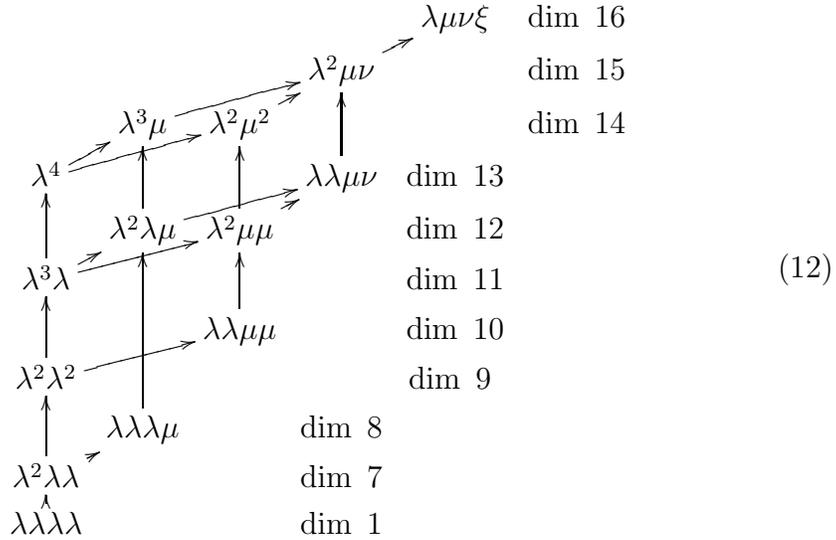

Let us compare \eqref{ky} and
\eqref{nkt}. The graph \eqref{ky} is
infinite; it is the disjoint union of
linear subgraphs that are obtained from
\begin{equation}\label{kud}
\lambda\lambda\lambda\lambda\to \lambda^2\lambda\lambda
\to\cdots\to\lambda^4,\quad
\lambda\lambda\lambda\mu \to\lambda^2\lambda\mu\to \lambda^3\mu,\
\dots,\ \lambda\mu\nu\xi
\end{equation}
by replacing their parameters by
unequal complex numbers (the numbers of
parameters in the vertices of the
linear subgraphs \eqref{kud} are equal
to 1, 2, 2, 3, 4, respectively). Thus,
although the sequences of Greek letters
in the vertices of \eqref{ky} and
\eqref{nkt} are the same, each vertex
of \eqref{ky} represents an infinite
set of similarity classes whose
matrices have the same Jordan type (and
so these similarity classes have the
same dimension), whereas the
corresponding vertex in \eqref{nkt}
represents only one bundle, which is
the union of these similarity classes;
its dimension is equal to the dimension
of any of its similarity classes plus
the number of parameters. Notice that
each arrow of \eqref{ky} corresponds to
an arrow of \eqref{nkt}, but
\eqref{nkt} has additional arrows.
\end{example}

\section{Perturbations of matrices
determined up to
congruence}\label{ssss}

Dmytryshyn, Futorny, and Sergeichuk
\cite{f_s} constructed miniversal
deformations of the following
congruence canonical matrices given by
Horn and Sergeichuk
\cite{hor-ser_transp,hor-ser_can}:

\begin{quote}
\emph{Every square
      complex matrix is congruent
      to a direct sum, determined
      uniquely up to permutation of
      summands, of matrices of the
      form
\begin{equation*}
\label{can}
\begin{bmatrix}0&I_m\\
J_m(\lambda) &0
\end{bmatrix},
 \qquad
\begin{bmatrix} 0&&&&
\udots
\\&&&-1&\udots
\\&&1&1\\ &-1&-1& &\\
1&1&&&0
\end{bmatrix},
 \qquad
J_k(0),
  \end{equation*}
in which $\lambda\in\mathbb
C\smallsetminus\{0,(-1)^{m+1}\}$ and is
determined up to replacement by
$\lambda^{-1}$. }
\end{quote}

The miniversal deformations
\cite[Theorem 2.2]{f_s} of congruence
canonical matrices are rather
cumbersome, so we give them only for
$2\times 2$ and $3\times 3$ matrices.

\begin{theorem}[{\cite[Example 2.1]{f_s}}]
\label{le_def_a} Let $A$ be any
$2\times 2$ or $3\times 3$ matrix. Then
all matrices $A+X$ that are
sufficiently close to $A$ can be
simultaneously reduced by some
transformation
\begin{equation}\label{dtg}
{\cal
S}(X)^T (A+X) {\cal
S}(X),\quad\text{${\cal S}(X)$ is holomorphic
at $0$,}
\end{equation}
to one of the following forms, in which
$\lambda \in \mathbb
C\smallsetminus\{-1,1\}$ and each
nonzero $\lambda $ is determined up to
replacement by $\lambda^{-1} $.

$\bullet$ If $A$ is $2\times 2$:
\begin{align*}
&\begin{bmatrix} 0&\\
&0
 \end{bmatrix}+
\begin{bmatrix}
 *&*\\ *&*
 \end{bmatrix},
                            &&
\begin{bmatrix} 1&\\
&0
 \end{bmatrix}+
\begin{bmatrix}
 0&0\\ *&*
 \end{bmatrix},
                         &&
\begin{bmatrix} 1&\\
&1
 \end{bmatrix}+
\begin{bmatrix}
 0&0\\ *&0
 \end{bmatrix},
                          \\&
\begin{bmatrix} 0&1\\
-1&0
 \end{bmatrix}+
\begin{bmatrix}
 *&0\\ *&*
 \end{bmatrix},
                       &&
\begin{bmatrix} 0&-1\\
1&1
 \end{bmatrix}+
\begin{bmatrix}
 *&0\\ 0&0
 \end{bmatrix},
                         &&
\begin{bmatrix} 0&1\\
\lambda &0
 \end{bmatrix}+
\begin{bmatrix}
 0&0\\ *&0
 \end{bmatrix}.
\end{align*}

$\bullet$  If $A$ is $3\times 3$:
\begin{longtable}{ll}
$\begin{bmatrix}
0&&\\
&0&\\&&0
 \end{bmatrix}+
\begin{bmatrix}
 *&*&*\\  *&*&*\\ *&*&*
 \end{bmatrix},$
                             &
$\begin{bmatrix} 1&&\\
&0&\\&&0
 \end{bmatrix}+
\begin{bmatrix}
 0&0&0\\  *&*&*\\ *&*&*
 \end{bmatrix},$\vspace{3pt}
                           \\
$\begin{bmatrix} 1&&\\
&1&\\&&0
 \end{bmatrix}+
\begin{bmatrix}
 0&0&0\\  *&0&0\\ *&*&*
 \end{bmatrix},$
                            &
$\begin{bmatrix} 1&&\\ &1&\\&&1
 \end{bmatrix}+
\begin{bmatrix}
 0&0&0\\  *&0&0\\
 *&*&0
 \end{bmatrix},$\vspace{3pt}
                       \\
$\begin{bmatrix} 0&1&\\
-1&0&\\&&0
 \end{bmatrix}+
\begin{bmatrix}
 *&0&0\\  *&*&0\\ *&*&*
 \end{bmatrix},$
                      &
$\begin{bmatrix} 0&1&\\ \lambda
&0&\\&&0
 \end{bmatrix}+
\begin{bmatrix}
 0&0&0\\  *&0&0\\
 *&*&*
 \end{bmatrix}
\ (\lambda\ne 0),$\vspace{3pt}
                        \\
$\begin{bmatrix} 0&1&\\ 0 &0&\\&&0
 \end{bmatrix}+
\begin{bmatrix}
 0&0&0\\  *&0&*\\
 *&0&*
 \end{bmatrix},$
                          &
$\begin{bmatrix} 0&-1&\\ 1&1&\\&&0
 \end{bmatrix}+
\begin{bmatrix}
 *&0&0\\  0&0&0\\ *&*&*
 \end{bmatrix},$\vspace{3pt}
                           \\
$\begin{bmatrix} 0&1&\\
-1&0&\\&&1
 \end{bmatrix}+
\begin{bmatrix}
 *&0&0\\  *&*&0\\
 0&0&0
 \end{bmatrix},$
                         &
$\begin{bmatrix} 0&1&\\ \lambda
&0&\\&&1
 \end{bmatrix}+
\begin{bmatrix}
 0&0&0\\  *&0&0\\
 0&0&0
 \end{bmatrix},$\vspace{3pt}
                   \\
$ \begin{bmatrix} 0&-1&\\ 1&1&\\&&1
 \end{bmatrix}+
\begin{bmatrix}
 *&0&0\\  0&0&0\\
 0&0&0
 \end{bmatrix},$
                      &
$\begin{bmatrix} 0&1&0\\ 0&0&1\\0&0&0
 \end{bmatrix}+
\begin{bmatrix}
 0&0&0\\  0&0&0\\
 *&0&*
 \end{bmatrix},$\vspace{3pt}
                     \\
$ \begin{bmatrix} 0&0&1\\
0&-1&-1\\1&1&0
 \end{bmatrix}+
\begin{bmatrix}
 0&0&0\\  *&0&0\\
 0&0&0
 \end{bmatrix}.$
\end{longtable}

Each of these matrices has the form
$A_{\rm can}+{\cal D}$ in which $A_{\rm
can}$ is a canonical matrix for
congruence and the stars in ${\cal D}$
are complex numbers that tend to zero
as $X$ tends to zero. The number of
stars is the smallest that can be
attained by using transformations
\eqref{dtg}; it is equal to the
codimension of the congruence class of
$A$.
\end{theorem}

The codimension of the congruence class
of a congruence canonical matrix
$A\in{\mathbb C}^{n\times n}$ was
calculated by Dmytryshyn, Futorny, and
Sergeichuk \cite{f_s} and independently
by  De Ter\'an and Dopico
\cite{ter-dor}; it is defined as
follows. For each small matrix
$X\in{\mathbb C}^{n\times n}$,
\[
(I+X)^TA(I+X)
=A+\underbrace{X^T A+ AX}
_{\text{small}}
+\underbrace{X^T AX}
_{\text{very small}}
\]
and so the congruence class of $A$ in a
small neighborhood of $A$ can be
obtained by a very small deformation of
the affine matrix space $\{A+ X^T A+
AX\,|\,X\in{\mathbb C}^{n\times n}\}$.
(By the local Lipschitz property
\cite{rodm}, if $A$ and $B$ are close
to each other and $B=S^{T}AS$ with a
nonsingular $S$, then $S$ can be taken
near $I_n$.)

The vector space
\[
T(A):=\{X^TA+AX\,|\,X\in{\mathbb
C}^{n\times n}\}
\]
is the tangent space to the congruence
class of $A$ at the point $A$. The
numbers
\[
\dim_{\mathbb C} T(A),\qquad \codim_{\mathbb C} T(A):= n^2-\dim_{\mathbb C} T(A)
\]
are called the \emph{dimension} and
\emph{codimension} of the congruence
class of $A$.

Congruence bundles are defined by
Futorny, Klimenko, and Sergeichuk
\cite{f-k-s_cov_congr} via bundles of
matrix pairs under equivalence. Recall,
that  pairs $(A,B)$ and $(A',B')$ of
$m\times n$ matrices are
\emph{equivalent} if there are
nonsingular $R$ and $S$ such that
$RAS=A'$ and $RBS=B'$. By Kronecker's
theorem about matrix pencils
\cite[Chapter XII, \S\,3]{gan}, each
pair $(A,B)$ of matrices of the same
size is equivalent to
\begin{equation}\label{lui}
{\cal L}\oplus {\cal P}_1(\lambda_1)
\oplus\dots\oplus {\cal P}_t(\lambda_t),
\quad \lambda_i\ne \lambda _j\text{ if }
i\ne j,\quad \lambda_1,\dots,\lambda_t
\in\mathbb C\cup\infty,
\end{equation}
in which ${\cal L}$ is a direct sum of
pairs of the form $(L_k,R_k)$ and
$(L_k^T,R_k^T)$, $k=1,2,\dots$, defined
by
\begin{equation*}\label{1.4}
L_k:=\begin{bmatrix}
1&0&&0\\&\ddots&\ddots&\\0&&1&0
\end{bmatrix},\quad
R_k:=\begin{bmatrix}
0&1&&0\\&\ddots&\ddots&\\0&&0&1
\end{bmatrix}\quad\text{($(k-1)$-by-$k$)},
\end{equation*}
and each ${\cal P}_i(\lambda_i)$ is a
direct sum of pairs of the form
\[
\text{$(I_k,J_k(\lambda_i))$ if $\lambda_i
\in \mathbb C$\quad or\quad $(J_k(0),I_k)$ if
$\lambda_i = \infty$.}
\]
The direct sums ${\cal L}$ and ${\cal
P}_i(\lambda_i)$ are determined by
$(A,B)$ uniquely, up to permutation of
summands. The \emph{equivalence bundle}
of \eqref{lui} consists of all matrix
pairs that are equivalent to pairs of
the form
\[
{\cal L}\oplus {\cal P}_1(\mu_1)
\oplus\dots\oplus {\cal P}_t(\mu_t),
\quad \mu_i\ne \mu_j\text{ if }
i\ne j,\quad \mu_1,\dots,\mu_t
\in\mathbb C\cup\infty,
\]
with the same ${\cal L},{\cal P}_1,
\dots, {\cal P}_t$ (see \cite{kag}).

The definition of bundles of matrices
under congruence is not so evident.
They could be defined via the
congruence canonical form by analogy
with bundles of matrices under
similarity and bundles of matrix pairs,
but, unlike the Jordan and Kronecker
canonical forms, the \emph{perturbation
behavior of a congruence canonical
matrix with parameters depends on the
values of its parameters}, which is
illustrated by the canonical matrices
$\left[\begin{smallmatrix} 0&1\\ -1&0&
\end{smallmatrix}\right]$ and $\left[\begin{smallmatrix}
0&1\\ \lambda&0
\end{smallmatrix}\right]$ in the
left graph in Figure \ref{kib4}.

\begin{definition}
[{\cite{f-k-s_cov_congr}}]\label{lqw}
Two square matrices $A$ and $B$ are in
the same \emph{congruence bundle} if
and only if the pairs $(A, A^T)$ and
$(B, B^T)$ are in the same equivalence
bundle.
\end{definition}

Definition \ref{lqw} is based on
Roiter's statement (see \cite[Lemma
4.1]{f-k-s_cov_congr}): two $n\times n$
matrices $A$ and $B$ are congruent if
and only if the pairs $(A, A^T)$ and
$(B, B^T)$ are equivalent.

\begin{example}\label{s5.1}
The closure graphs for congruence
classes and congruence bundles of
$2\times 2$ matrices are presented in
Figure \ref{kib4}; they were
constructed by Futorny, Klimenko, and
Sergeichuk \cite{f-k-s_cov_congr}.
\begin{figure}[hbt]
\[
\raisebox{-55pt}{\xymatrix@R=3pt@C=6pt{
{\begin{bmatrix} 0&-1\\ 1&1
 \end{bmatrix}}
    &
{\begin{bmatrix} 0&1\\ {\lambda} &0
\end{bmatrix}}
   &{\begin{bmatrix} 1&\\
&1 \end{bmatrix}}
    \\ \\
  &{\begin{bmatrix}
1&\\ &0
 \end{bmatrix}}\ar[uu] \ar[luu]
\ar[uur]&
    \\
{\begin{bmatrix} 0&1\\
-1&0
 \end{bmatrix}}\ar[uuu]& &
      \\
&{\begin{bmatrix} 0&\\
&0
 \end{bmatrix}}
\ar[ul]\ar[uu]&}}
            %%%%%%%%%%%%%%%%%%%%%%%%%%%%%%%%%%%%
            %%%%%%%%%%%%%%%%%%%%%%%%%%%%%%%%%%%%
            %%%%%%%%%%%%%%%%%%%%%%%%%%%%%%%%%%%%
            %%%%%%%%%%%%%%%%%%%%%%%%%%%%%%%%%%%%
 \xymatrix@R=1pt@C=-3pt{
{\phantom{{}_{\lambda \ne\pm
1}}\left\{\begin{bmatrix} 0&1\\
{\lambda}
&0
 \end{bmatrix}\right\}_{\lambda}}
 &&&{\qquad\dim\ 4}
 %%%%%%%%%%%%%%%%%
\\    \\ \\
  %%%%%%%%%%%%%%%%%
{\begin{bmatrix} 0&-1\\ 1&1
 \end{bmatrix}}\ar[uuu]
    &
{\begin{bmatrix} 1&\\
&1
\end{bmatrix}}\ar[uuul]&&{\qquad\dim\ 3}
  %%%%%%%%%%%%%%%%%
 \\   \\ \\
  %%%%%%%%%%%%%%%%%
  &{\begin{bmatrix}
1&\\ &0
 \end{bmatrix}}\ar[uuu] \ar[luuu]
 &&{\qquad\dim\ 2}
   %%%%%%%%%%%%%%%%%%
    \\
   %%%%%%%%%%%%%%%%%%
{\begin{bmatrix} 0&1\\
-1&0
 \end{bmatrix}}\ar[uuu]&
 &&{\qquad\dim\ 1}
   %%%%%%%%%%%%%%%%%%
      \\
   %%%%%%%%%%%%%%%%%%
&{\begin{bmatrix} 0&\\
&0
 \end{bmatrix}}
\ar[ul]\ar[uu]&&{\qquad\dim\ 0}
}
\]
\caption{\small\it The closure graphs for congruence
classes and congruence bundles of
$2\times 2$ matrices, in which $\lambda \in \mathbb C
\smallsetminus\{-1,1\}$ and each nonzero
$\lambda $ is determined up to
replacement by $\lambda^{-1} $.}
\label{kib4}
\end{figure}
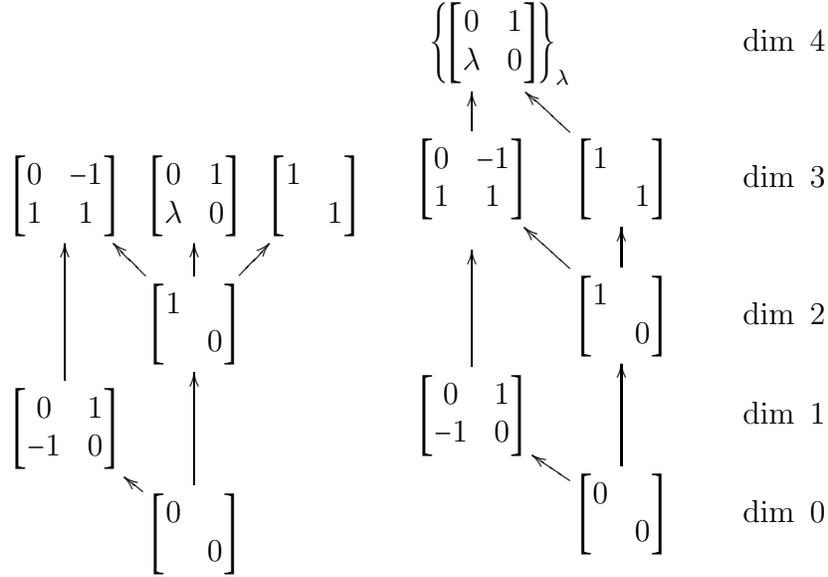

\begin{description}
  \item[The left graph] in Figure
      \ref{kib4} is the closure
      graph for congruence classes
      of ${2\times 2}$ matrices.
      The congruence classes are
      given by their $2\times 2$
      canonical matrices for
      congruence. The graph is
      infinite:
$\left[\begin{smallmatrix} 0&1\\
\lambda&0
 \end{smallmatrix}\right]$ represents the
infinite set of vertices indexed by
$\lambda \in \mathbb C
\smallsetminus\{-1,1\}$.

  \item[The right graph] is the
      closure graph for congruence
      bundles of ${2\times 2}$
      matrices. The vertex $\left\{
      \left[
\begin{smallmatrix} 0&1\\
\lambda &0
 \end{smallmatrix}\right]
 \right\}_{\lambda}$ represents the
bundle that consists of all
matrices whose congruence canonical
forms are
$\left[\begin{smallmatrix} 0&1\\
\lambda &0
 \end{smallmatrix}\right]$ with $\lambda
 \ne \pm 1$. The other vertices are
canonical matrices; their bundles
coincide with their congruence
classes. Note that $ \left[
\begin{smallmatrix} 0&1\\
-1 &0
 \end{smallmatrix}\right]$ and $\left[
\begin{smallmatrix} 0&1\\
\lambda &0
 \end{smallmatrix}\right]$  ($\lambda
 \ne \pm 1$) properly belong to
distinct bundles because these
matrices have distinct properties
 with respect to perturbations,
which is illustrated by the left
graph. Other arguments in favor of
Definition \ref{lqw} of congruence
bundles are given in \cite[Section
6] {f-k-s_cov_congr}.
\end{description}
The congruence classes and bundles with
vertices on the same horizontal level
have the same dimension, which is
indicated to the right.
\end{example}

\begin{example}\label{s5.2}
The closure graphs for congruence
classes and congruence bundles of
${3\times 3}$ matrices are presented in
Figure \ref{kib5}. They were
constructed by Futorny, Klimenko, and
Sergeichuk \cite{f-k-s_cov_congr}.
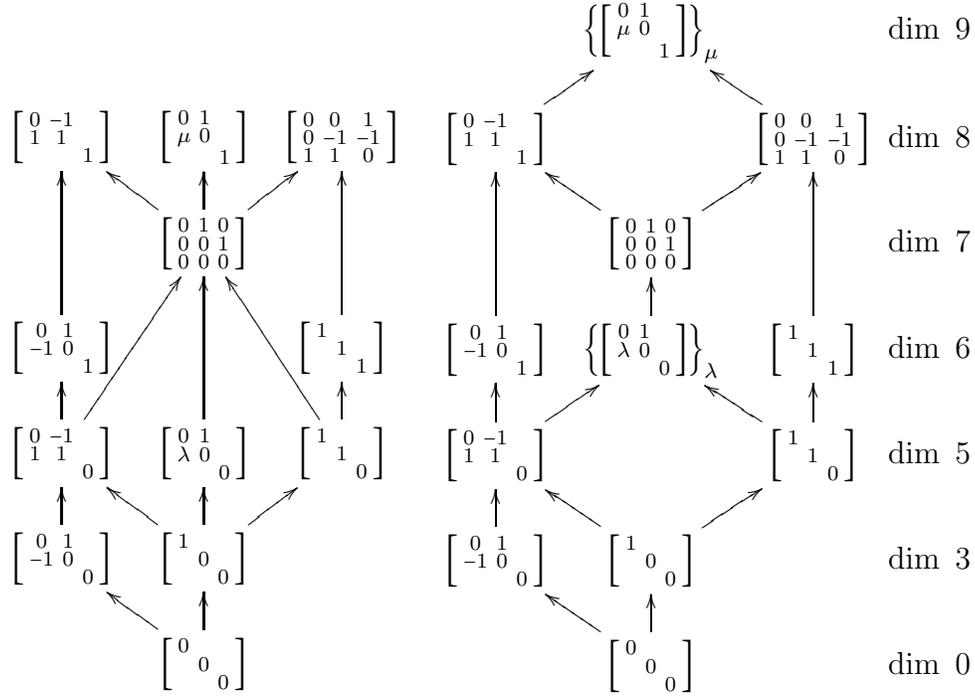
\begin{figure}[hbt]
\[
\xymatrix@C=1pt@R=14pt{
{\phantom{\left[\begin{smallmatrix} 0&-1&\\
1&1&\\&&A^A
\end{smallmatrix}\right]
}}&&&&\\
{\left[\begin{smallmatrix} 0&-1&\\
1&1&\\&&1
\end{smallmatrix}\right]}&&
{\left[\begin{smallmatrix}
0&1&\\ {\mu} &0&\\&&1
\end{smallmatrix}\right]} &&
{\left[\begin{smallmatrix}
0&0&1\\ 0&-1&-1\\1&1&0
\end{smallmatrix}\right]}
    \\
&& {\left[\begin{smallmatrix}
0&1&0\\ 0&0&1\\0&0&0
\end{smallmatrix}\right]} \ar[urr]\ar[ull] \ar[u]&&
    \\
{\left[\begin{smallmatrix}
0&1&\\ -1&0&\\&&1
\end{smallmatrix}\right]}\ar[uu]&&&&
{\left[\begin{smallmatrix}
1&&\\ &1&\\&&1
\end{smallmatrix}\right]}\ar[uu]\\
{\left[\begin{smallmatrix}
0&-1&\\ 1&1&\\&&0
\end{smallmatrix}\right]}
\ar[uurr]\ar[u]&
&{\left[\begin{smallmatrix}
0&1&\\ {\lambda}
&0&\\&&0
\end{smallmatrix}\right]}\ar[uu]
&&{\left[\begin{smallmatrix}
1&&\\ &1&\\&&0
\end{smallmatrix}\right]} \ar[uull]\ar[u]
         \\
{\left[\begin{smallmatrix}
0&1&\\ -1&0&\\&&0
\end{smallmatrix}\right]}
%\ar@<1ex>@(ul,dl)[uu]
\ar[u] &&
{\left[\begin{smallmatrix}
1&&\\ &0&\\&&0
\end{smallmatrix}\right]}\ar[u]
\ar[ull]%\ar[uull]
\ar[urr]&&
 \\
&&{\left[\begin{smallmatrix}
0&&\\ &0&\\&&0
 \end{smallmatrix}\right]}\ar[ull]\ar[u]&&
      }
            %           &&&&&&&&&&&&&&&&&&&&&
\quad\xymatrix@C=1pt@R=14pt{
&&
{\left\{\left[\begin{smallmatrix}
0&1&\\ {\mu} &0&\\&&1
\end{smallmatrix}\right]\right\}_{\mu}} &&
&{\dim\ 9}
\\
{\left[\begin{smallmatrix}
0&-1&\\ 1&1&\\&&1
\end{smallmatrix}\right]}\ar[urr]&&&&
{\left[\begin{smallmatrix}
0&0&1\\ 0&-1&-1\\1&1&0
\end{smallmatrix}\right]}\ar[ull]&
{\dim\ 8}
    \\
&& {\left[\begin{smallmatrix}
0&1&0\\ 0&0&1\\0&0&0
\end{smallmatrix}\right]} \ar[urr]\ar[ull]&& &
{\dim\ 7}
    \\
{\left[\begin{smallmatrix}
0&1&\\ -1&0&\\&&1
\end{smallmatrix}\right]}\ar[uu]&&
{\left\{\left[\begin{smallmatrix}
0&1&\\ {\lambda}
&0&\\&&0
\end{smallmatrix}\right]\right\}_{\lambda}}\ar[u]&&
{\left[\begin{smallmatrix}
1&&\\ &1&\\&&1
\end{smallmatrix}\right]}\ar[uu]&
{\dim\ 6}\\
{\left[\begin{smallmatrix}
0&-1&\\ 1&1&\\&&0
\end{smallmatrix}\right]}\ar[urr]
\ar[u]&
&
&&{\left[\begin{smallmatrix}
1&&\\ &1&\\&&0
\end{smallmatrix}\right]} \ar[ull]\ar[u] &
{\dim\ 5}
         \\
{\left[\begin{smallmatrix}
0&1&\\ -1&0&\\&&0
\end{smallmatrix}\right]}
%\ar@<1ex>@(ul,dl)[uu]
\ar[u] &&
{\left[\begin{smallmatrix}
1&&\\ &0&\\&&0
\end{smallmatrix}\right]}
\ar[ull]%\ar[uull]
\ar[urr]&&
 &
{\dim\ 3}\\
&&{\left[\begin{smallmatrix}
0&&\\ &0&\\&&0
 \end{smallmatrix}\right]}\ar[ull]\ar[u]&&
 &
{\dim\ 0}
      }
\]
\caption{\small\it The closure graphs for congruence
classes and congruence bundles of
${3\times 3}$ matrices, in which
$\lambda,\mu\ne\pm 1$, and nonzero $\lambda
$ and $\mu $ are determined up to
replacements by $\lambda^{-1} $ and
$\mu^{-1}$.} \label{kib5}
\end{figure}
\begin{description}
  \item[The left graph] in Figure
      \ref{kib5} is the closure
      graph for congruence classes
      of ${3\times 3}$ matrices.
      The congruence classes are
      given by their $3\times 3$
      canonical matrices for
      congruence. The graph is
      infinite:
$\left[\begin{smallmatrix} 0&1&\\
{\lambda} &0&\\&&0
\end{smallmatrix}\right]$ and
$\left[\begin{smallmatrix} 0&1&\\
{\mu} &0&\\&&1
\end{smallmatrix}\right]$ represent the
infinite sets of vertices indexed
by $\lambda,\mu\ne\pm 1$.

  \item[The right graph] is the
      closure graph for congruence
      bundles of ${3\times 3}$
      matrices. The vertices
$\left\{\left[\begin{smallmatrix} 0&1&\\
{\lambda} &0&\\&&0
\end{smallmatrix}\right]\right\}_{\lambda}$
 and
$\left\{\left[\begin{smallmatrix}
0&1&\\ {\mu} &0&\\&&1
\end{smallmatrix}\right]\right\}_{\mu}$
represent the bundles that consist
of all matrices whose congruence
canonical forms are
$\left[\begin{smallmatrix} 0&1&\\
{\lambda} &0&\\&&0
\end{smallmatrix}\right]$
$(\lambda\ne \pm 1$) or
$\left[\begin{smallmatrix} 0&1&\\
{\mu} &0&\\&&1
\end{smallmatrix}\right]$
($\mu \ne \pm 1$), respectively.
The other vertices are canonical
matrices; their bundles coincide
with their congruence classes.
\end{description}

\end{example}

\begin{remark}
Let $M$ be a $2\times 2$ or $3\times 3$
canonical matrix for congruence.

\begin{itemize}
\item Let $N$ be another canonical
    matrix for congruence of the
    same size. Each neighborhood of
$M$ contains a matrix from the
congruence class (respectively,
bundle) of $N$ if and only if there
is a directed path from $M$ to $N$
in the left (resp. right) graph in
Figures \ref{kib4} or \ref{kib5}.
Note that there always exists the
``lazy'' path of length $0$ from
$M$ to $M$ if $M=N$.

  \item The closure of the
      congruence class (resp.
      bundle) of $M$
is equal to the union of the
      congruence classes (resp.
      bundles) of all canonical
      matrices $N$ such that there
      is a directed path from $N$
      to $M$.
\end{itemize}
\end{remark}

\section{Perturbations of matrices determined up
*congruence}\label{mmj}

Dmytryshyn, Futorny, and Sergeichuk
\cite{def-sesq} constructed miniversal
deformations of the following
*congruence canonical matrices given by
Horn and Sergeichuk
\cite{hor-ser_transp,hor-ser_can}:

\begin{quote}
\emph{Every square complex matrix is
*congruent to a direct sum, determined
uniquely up to permutation of summands,
of matrices of the form
\begin{equation}\label{tabll}
\begin{bmatrix}0&I_m\\
J_m(\lambda) &0
\end{bmatrix},
\qquad \mu \begin{bmatrix}
0 &  &  & 1\\
&  &
\udots & i\\
& 1 &
\udots & \\
1 & i &  & 0
\end{bmatrix},\qquad
J_k(0),
\end{equation}
in which $\lambda,\mu \in\mathbb C$,
$|\lambda |>1$, and $|\mu |=1$. } (The
condition $|\lambda |>1$ can be
replaced by $0<|\lambda |<1$.)
\end{quote}

The miniversal deformations
\cite[Theorem 2.2]{def-sesq} of
*congruence canonical matrices are
rather cumbersome, so we give them only
for $2\times 2$ and $3\times 3$
matrices.

\begin{theorem}
\label{le_de} Let $A$ be any $2\times
2$ or $3\times 3$ matrix. Then all
matrices $A+X$ that are sufficiently
close to $A$ can be simultaneously
reduced by some transformation
\[
{\cal
S}(X)^* (A+X) {\cal
S}(X),\quad
\begin{matrix}
\text{${\cal S}(X)$
is nonsingular and conti-}\\
\text{nuous on a neighborhood of zero,}
\end{matrix}
\]
to one of the following forms.
\smallskip

$\bullet$ If $A$ is $2\times
      2$:
\begin{align*}
&\begin{bmatrix} 0&0\\
0&0
 \end{bmatrix}+
\begin{bmatrix}
 *&*\\ *&*
 \end{bmatrix},
                            &&
\begin{bmatrix} \mu_1 &0\\
0&0
 \end{bmatrix}+
\begin{bmatrix}
\varepsilon_1&0\\ *&*
 \end{bmatrix},
                         &&
\begin{bmatrix} \mu_1 &0\\
0&\mu_2
 \end{bmatrix}+
\begin{bmatrix}
\varepsilon_1&0\\ \delta_{21}&
\varepsilon_2
 \end{bmatrix},
                         \\&
\begin{bmatrix} 0&\mu_1 \\
\mu_1  &i\mu_1
 \end{bmatrix}+
\begin{bmatrix}
 *&0\\ 0&0
 \end{bmatrix},
                         &&
\begin{bmatrix} 0&1\\
\lambda  &0
 \end{bmatrix}+
\begin{bmatrix}
 0&0\\ *&0
 \end{bmatrix}.
\end{align*}

%%%%%%%%%%%%%%%%%%%%%%%%%%%%%

$\bullet$ If $A$ is $3\times
      3$:
\begin{longtable}{ll}
$\begin{bmatrix} 0\\
&0\\&&0
 \end{bmatrix}+
\begin{bmatrix}
 *&*&*\\ *&*&*\\ *&*&*
 \end{bmatrix},$
                         &
$\begin{bmatrix} \mu_1\\
&0\\&&0
 \end{bmatrix}+
\begin{bmatrix}
\varepsilon_1&0&0\\ *&*&*\\ *&*&*
 \end{bmatrix},$\vspace{3pt}
                            \\
$\begin{bmatrix} \mu_1\\
&\mu_2\\&&0
 \end{bmatrix}+
\begin{bmatrix}
\varepsilon_1&0&0\\ \delta_{21}&
\varepsilon_2&0\\ *&*&*
 \end{bmatrix},$
                           &
$\begin{bmatrix} \mu_1\\
&\mu_2\\&&\mu_3
 \end{bmatrix}+
\begin{bmatrix}
\varepsilon_1&0&0\\ \delta_{21}&
\varepsilon_2&0\\
\delta_{31}&\delta_{32}& \varepsilon_3
 \end{bmatrix},$\vspace{3pt}
                               \\
$\begin{bmatrix} 0&\mu_1 \\
\mu_1  &i\mu_1\\&&\mu _2
 \end{bmatrix}+
\begin{bmatrix}
 *&0&0\\ 0&0&0\\
\delta _{21}&0&\varepsilon_2
 \end{bmatrix},$
                            &
$\begin{bmatrix} 0&\mu_1 \\
\mu_1  &i\mu_1\\&&0
 \end{bmatrix}+
\begin{bmatrix}
 *&0&0\\ 0&0&0\\
*&*&*
 \end{bmatrix},$\vspace{3pt}
                          \\
$\begin{bmatrix} 0&1 \\
\lambda &0 \\&&\mu _1
 \end{bmatrix}+
\begin{bmatrix}
 0&0&0\\ *&0&0\\
0&0&\varepsilon_1
 \end{bmatrix},$
                         &
$\begin{bmatrix} 0&1 \\
\lambda &0 \\&&0
 \end{bmatrix}+
\begin{bmatrix}
 0&0&0\\ *&0&0\\
*&*&*
 \end{bmatrix}\ (\lambda\ne 0),$\vspace{3pt}
                            \\
$\begin{bmatrix} 0&1\\0&0\\ &&0
 \end{bmatrix}+
\begin{bmatrix} 0&0&0\\
*&0&*\\ *&0&*
 \end{bmatrix},$
                        &
$\begin{bmatrix} 0&1&0\\
0&0&1\\0&0&0
 \end{bmatrix}+
\begin{bmatrix} 0&0&0\\
0&0&0\\ *&0&*
 \end{bmatrix},$\vspace{3pt}
                         \\
$\begin{bmatrix} 0&0&\mu_1 \\
0&\mu_1  &i\mu_1\\ \mu_1  &i\mu_1&0\\
 \end{bmatrix}+
\begin{bmatrix}
 0&0&0\\0&\varepsilon_1&0\\ 0&0&0
\end{bmatrix},\quad$
\end{longtable}
Each of these matrices has the form
$A_{\rm can}+{\cal D}$, in which
$A_{\rm can}$ is a canonical matrix for
*congruence, the stars in ${\cal D}$
are complex numbers, $|\lambda| <1$,
$|\mu_1|=|\mu_2|=|\mu_3|=1$, and
\begin{align*}\label{omr}
&\varepsilon_l\in\mathbb R\text{ if }
\mu_l\notin\mathbb R
                    &&
\delta_{lr}=0\text{ if } \mu_l\ne\pm\mu_r
 \\
&\varepsilon_l\in i\mathbb R\text{ if }
\mu_l\in\mathbb R
                    &&
\delta_{lr}\in\mathbb C\text{ if } \mu_l=\pm\mu_r
\end{align*}
$($Clearly, ${\cal D}$ tends to zero as
$X$ tends to zero.$)$ For each $A_{\rm
can}+{\cal D}$, twice the number of its
stars plus the number of its entries
$\varepsilon_{l},\delta_{lr}$ is equal
to the codimension over $\mathbb R$ of
the *congruence class of $A_{\rm can}$.
\end{theorem}

The codimension of the *congruence
class of a *congruence canonical matrix
$A\in{\mathbb C}^{n\times n}$ was
calculated by De Ter\'an and Dopico
\cite{t_d_*} and independently by
Dmytryshyn, Futorny, and Sergeichuk
\cite{def-sesq}; it is defined as
follows. For each $A\in{\mathbb
C}^{n\times n}$ and a small matrix
$X\in{\mathbb C}^{n\times n}$,
\[
(I+X)^*A(I+X)
=A+\underbrace{X^* A+ AX}
_{\text{small}}
+\underbrace{X^* AX}
_{\text{very small}}
\]
and so the *congruence class of $A$ in
a small neighborhood of $A$ can be
obtained  by a very small deformation
of the real affine matrix space $\{A+
X^* A+ AX\,|\,X\in{\mathbb C}^{n\times
n}\}$. (By the local Lipschitz property
\cite{rodm}, if $A$ and $B$ are close
to each other and $B=S^*AS$ with a
nonsingular $S$, then $S$ can be taken
near $I_n$). The real vector space
\[
T(A):=\{X^*A+AX\,|\,X\in{\mathbb
C}^{n\times n}\}
\]
is the tangent space to the *congruence
class of $A$ at the point $A$.  The
numbers
\[
\dim_{\mathbb R} T(A),\qquad \codim_{\mathbb R} T(A):=2n^2-\dim_{\mathbb R} T(A)
\]
are called the \emph{dimension} and,
respectively, \emph{codimension over
${\mathbb R}$} of the *congruence class
of $A$.

\begin{example}\label{thu}
The closure graph for *congruence
classes of ${2\times 2}$ matrices is
presented in Figure \ref{kib6}; it was
constructed by Futorny, Klimenko, and
Sergeichuk \cite{f-k-s_cov_*congr}.
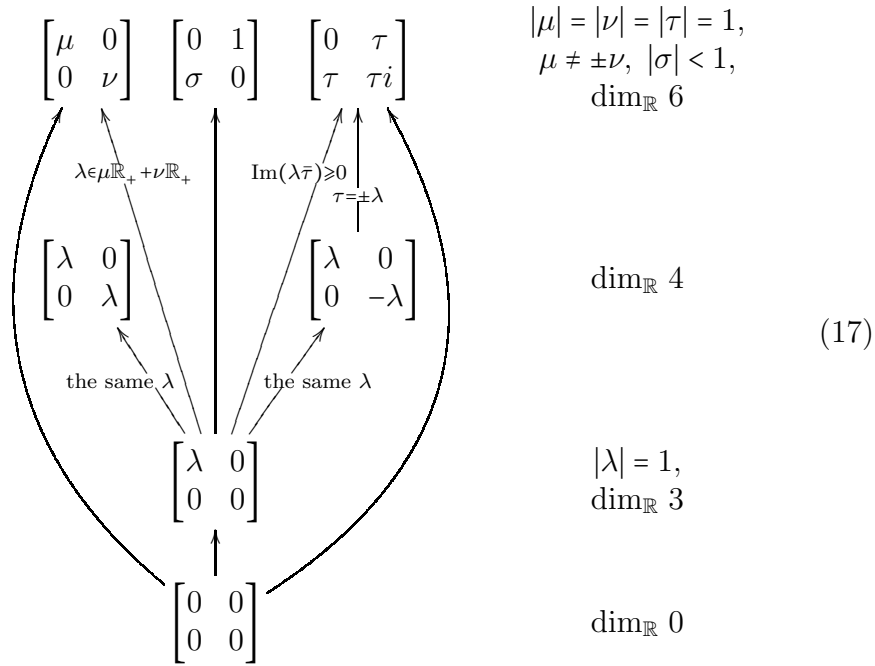
\begin{figure}[hbt]
\begin{equation}\label{g1}
\begin{split}\qquad\qquad
\xymatrix@R=17pt@C=8pt{
{\begin{bmatrix} \mu &0\\ 0&\nu
 \end{bmatrix}}
           &
{\begin{bmatrix} 0&1\\ \sigma  &0
\end{bmatrix}}
           &
{\begin{bmatrix} 0&\tau \\
\tau &\tau i \end{bmatrix}}
 &&& {\begin{matrix}
 |\mu|=|\nu|=|\tau|=1,\\
        \mu\ne\pm\nu,\
|\sigma
      |<1,\\ \dim_{\mathbb R}\, 6
      \end{matrix}
}
            %%%%%%%%%%%%%%%%%%%%%%
     \\ \\
            %%%%%%%%%%%%%%%%%%%%%                                   &
{\begin{bmatrix} \lambda &0\\ 0&\lambda
 \end{bmatrix}}
&&
{\ \begin{bmatrix} \lambda &0\\ 0&-\lambda
 \end{bmatrix}}
\ar[uu]|-(.3){\tau =\pm\lambda}
   &&&
{\dim_{\mathbb R}\, 4}
            %%%%%%%%%%%%%%%%%%%%%%
     \\ \\
            %%%%%%%%%%%%%%%%%%%%%
&{\begin{bmatrix} \lambda &0\\
0&0
 \end{bmatrix}}
   \ar[uul]|-{\text{\rm the same }\lambda\quad\quad}
   \ar[uur]|-{\quad\quad\text{\rm the same }\lambda}
\ar[uuuu]\ar[uuuul]|-(.8){\ \ \:\lambda \in
\mu\mathbb R_{_+} +\nu\mathbb
R_{_+}}
\ar[uuuur]|-(.8){{\rm Im}
(\lambda\bar\tau)\ge 0\ \quad}& & &&
{\begin{matrix}
   |\lambda |=1,\\
\dim_{\mathbb R}\, 3
 \end{matrix}
}
            %%%%%%%%%%%%%%%%%%%%%%
     \\
            %%%%%%%%%%%%%%%%%%%%%
&{\begin{bmatrix} 0&0\\
0&0
 \end{bmatrix}}
\ar[u]\ar@/^4.2pc/[uuuuul]
\ar@/_5pc/[uuuuur]
& & && {\dim_{\mathbb R}\, 0}
} \end{split}
\end{equation}
\caption{\small\it The closure graph for *congruence
classes of ${2\times 2}$ matrices,
in which
$\mathbb R_+$ denotes the set of
nonnegative real numbers,
$\im (c)$ denotes the imaginary part
of $c\in\mathbb C$, and $\lambda ,\mu ,\nu ,\sigma
,\tau \in\mathbb C$.}
\label{kib6}
\end{figure}
Each *congruence class is given by its
canonical matrix, which is a direct sum
of blocks of the form \eqref{tabll}.
The graph is infinite: each vertex
except for $\left[\begin{smallmatrix}
0&0\\0&0
\end{smallmatrix}\right]$ represents
an infinite set of vertices indexed by
the parameters of the corresponding
canonical matrix. The *congruence
classes of canonical matrices that are
located at the same horizontal level in
\eqref{g1} have the same dimension over
$\mathbb R$, which is indicated to the
right. The arrow
$\left[\begin{smallmatrix} \lambda &0\\
0&0
\end{smallmatrix}\right]\to
\left[\begin{smallmatrix} \mu &0\\
0&\nu
\end{smallmatrix}\right]$ exists
if and only if $ \lambda
 =\mu a +\nu b$ for some nonnegative
$a,b\in\mathbb R$. The arrow
$\left[\begin{smallmatrix} \lambda &0\\
0&0
\end{smallmatrix}\right]\to
\left[\begin{smallmatrix} 0&\tau \\
\tau &i\tau
\end{smallmatrix}\right]$ exists
if and only if the imaginary part of
$\lambda\bar\tau$ is nonnegative. The
arrow
$\left[\begin{smallmatrix} \lambda &0\\
0&-\lambda
\end{smallmatrix}\right]\to
\left[\begin{smallmatrix} 0&\tau \\
\tau &i\tau
\end{smallmatrix}\right]$ exists
if and only if $\tau =\pm\lambda$. The
arrows
$\left[\begin{smallmatrix} \lambda &0\\
0&0
\end{smallmatrix}\right]\to
\left[\begin{smallmatrix} \lambda &0 \\
0&\pm\lambda
\end{smallmatrix}\right]$ exist
if and only if the value of $\lambda$
is the same in both matrices.  The
other arrows exist for all values of
parameters of their matrices.
\end{example}

\begin{remark}\rm
Let $M$ be a $2\times 2$ canonical
matrix for *congruence.

\begin{itemize}
\item Let $N$ be another $2\times
    2$ canonical matrix for
    *congruence. Each neighborhood
    of $M$ contains a matrix that
    is *congruent to
$N$ if and only if there is a
directed path
    from $M$ to $N$ in \eqref{g1}
    (if $M=N$, then there always
    exists the ``lazy'' path of
    length $0$ from $M$ to $N$).

  \item The closure of the
      *congruence class of $M$ is
      equal to the union of the
      *congruence classes of all
      canonical matrices $N$ such
      that there is a directed path
      from $N$ to $M$.
\end{itemize}
\end{remark}


\begin{thebibliography}{99}
%\end{thebibliography}

\bibitem{arn} V.I. Arnold, On matrices
    depending on parameters, {Russian Math. Surveys} 26 (2)
    (1971) 29--43.

\bibitem{arn2} V.I. Arnold, Lectures on
    bifurcations in versal
    families, {Russian Math.
    Surveys}
    27 (5) (1972) 54--123.

\bibitem{arn3} V.I. Arnold,
    {Geometrical Methods in the Theory
    of Ordinary Differential
    Equations}, Springer-Verlag, 1988.

\bibitem{ter-dor} F. De Ter\'an,
    F.M. Dopico, The
    solution of the equation $XA +
    AX^T= 0$ and its application to the
    theory of orbits, {Linear Algebra
    Appl.} 434 (2011) 44--67.

\bibitem{t_d_*} F. De Ter\'an, F.M.
    Dopico, The equation $XA + AX^*= 0$
    and the dimension of *congruence
    orbits, {Electr. J. Linear
    Algebra} 22 (2011) 448--465.

\bibitem{den-thi} H. den Boer,
    G.Ph.A. Thijsse,
    Semi-stability of sums of partial
    multiplicities under additive
    perturbation, Integral Equations
    Operator Theory 3 (1980)
    23--42.

\bibitem{f_s}  A.R. Dmytryshyn, V.
    Futorny, V.V.
    Sergeichuk, Miniversal deformations
    of matrices of bilinear forms,
    Linear Algebra Appl. 436 (2012)
    2670--2700. (Preliminary version: {Preprint RT-MAT 2007-04},
    Universidade de S\~ao Paulo, 2007, 34
    p.)

\bibitem{def-sesq} A.R. Dmytryshyn, V.
    Futorny, V.V. Sergeichuk,
    Miniversal deformations of matrices
    under *congruence and reducing
    transformations, arXiv:1105.2160.

\bibitem{kag} A. Edelman, E. Elmroth,
    B. K\r{a}gstr\"{o}m, A geometric
    approach to perturbation theory of
    matrices and matrix pencils. Part
    I: Versal deformations, {SIAM
    J. Matrix Anal. Appl.} 18
    (1997) 653--692.

\bibitem{kag2} A. Edelman, E.
    Elmroth, B. K\r{a}gstr\"{o}m,
    A geometric approach to
    perturbation theory of matrices and
    matrix pencils. Part II: A
    stratification-enhanced staircase
    algorithm, SIAM J. Matrix Anal.
    Appl. 20 (1999) 667--699.

\bibitem{e-j-k} E. Elmroth, P.
    Johansson, B.
    K\r{a}gstr\"{o}m, Computation
    and presentation of graph
    displaying closure hierarchies of
    Jordan and Kronecker structures,
    Numer. Linear Algebra Appl. 8
    (2001) 381--399.

\bibitem{f-k-s_cov_congr} V.
    Futorny, L. Klimenko,
    V.V. Sergeichuk, Change of the
    congruence canonical form of 2-by-2 and
3-by-3 matrices under perturbations,
arXiv:1004.3590.

\bibitem{f-k-s_cov_*congr} V.
    Futorny, L. Klimenko,
V.V. Sergeichuk, Change of the
*congruence canonical form of 2-by-2
matrices under perturbations,
arXiv:1304.5762.


\bibitem{gan} F.R. Gantmacher, The
    Theory of Matrices, Vol. 1 and 2,
    Chelsea, New York, 1959.

\bibitem{gar_ser} M.I. Garcia-Planas,
    V.V. Sergeichuk, Simplest
    miniversal deformations of
    matrices, matrix pencils, and
    contragredient matrix pencils,
    {Linear Algebra Appl.} 302--303
    (1999) 45--61.

\bibitem{hor-ser_transp} R.A.
    Horn, V.V. Sergeichuk, Congruence of a
    square matrix and its transpose,
    {Linear Algebra Appl.} 389 (2004)
    347--353.


\bibitem{hor-ser_can} R.A. Horn,
    V.V.
    Sergeichuk.
    Canonical forms for complex matrix
    congruence and *congruence, {Linear
    Algebra Appl.} 416 (2006) 1010--1032.

\bibitem{joh} P. Johansson, StratiGraph
    User's Guide, Technical Report
    UMINF 03.21 (ISSN-0348-0542),
    Department of
    Computing Science, Ume\r{a}
    University, Sweden, 2003. Available
    at: http://www8.cs.umu.se/~pedher/research/papers/sg-usersguide.pdf


\bibitem{k-s_triang} L.
    Klimenko, V.V. Sergeichuk, Block triangular
    miniversal deformations of matrices
    and matrix pencils,  in:  V. Olshevsky,  E.
    Tyrtyshnikov  (Eds), {Matrix
    Methods: Theory,  Algorithms and
    Applications},  World  Scientific
    Publishing  Co.  Pte.  Ltd.,
    Hackensack, NJ, 2010, pp. 69--84.

\bibitem{mai} A.A. Mailybaev,
    Transformation of families of
    matrices to normal forms and its
    application to stability theory,
    SIAM J. Matrix Anal. Appl. 21
    (1999) 396--417.

\bibitem{mai1} A.A. Mailybaev,
    Transformation to versal
    deformations of matrices, Linear
    Algebra Appl. 337
    (2001) 87--108.

\bibitem{mar-par} A.S. Markus,
    E.\`E. Parilis, The change
    of the Jordan structure of a matrix
    under small perturbations, Mat.
    Issled. 54 (1980) 98--109.
    English translation: Linear Algebra
    Appl. 54 (1983) 139--152.

\bibitem{pok} A. Pokrzywa, On
    perturbations and the
    equivalence orbit of a matrix
    pencil, Linear Algebra Appl. 82
    (1986) 99--121.

\bibitem{rodm} L. Rodman, Remarks on
    Lipschitz properties of matrix
    groups actions, {Linear Algebra
    Appl.} 434 (2011) 1513--1524.

\bibitem{wil} J.H. Wilkinson, The
    Algebraic Eigenvalue Problem, Oxford
    University Press, 1965.

\end{thebibliography}
\end{document}